\documentclass[12pt]{amsart}
\usepackage{amssymb}
\usepackage{enumerate,color}
\usepackage{graphicx}
\usepackage[text={6in,9in},centering]{geometry}
\usepackage{subfigure}
\usepackage{epstopdf}

\theoremstyle{plain}
  \newtheorem{thm}{Theorem}[section]
  \newtheorem{cor}[thm]{Corollary}
  \newtheorem{lem}[thm]{Lemma}
  
\theoremstyle{definition}
  \newtheorem{defn}{Definition}
  \newtheorem{exmp}{Example}
\theoremstyle{remark}

\def\R{\mathbb{R}}
\def\D{\mathcal{D}}

\def\card{\textrm{card}}

\newcommand{\bw}{\mathbf{w}}

\newcommand\hdd{\mbox{\rm dim}_{\rm H}\,} 
\newcommand\lcd{\mbox{\rm dim}_{\rm loc}\,} 

\begin{document}
\title[Multifractal analysis of  a class of self-affine Moran sets ]{ Multifractal analysis of a class of  self-affine Moran sets }
\author{Yifei Gu}
\address{Department of Mathematics, East China Normal University, No. 500, Dongchuan Road, Shanghai 200241, P. R. China}

\email{52275500012@stu.ecnu.edu.cn}

\author{Chuanyan Hou}
\address{College of Mathematics Sciences, Xinjiang Normal University, Urumqi, Xinjiang, 830054, P. R. China}
\email{hchy\_e@163.com}

\author{Jun Jie Miao}
\address{Department of Mathematics, East China Normal University, No. 500, Dongchuan Road, Shanghai 200241, P. R. China}

\email{jjmiao@math.ecnu.edu.cn}


\begin{abstract}

In the paper, we investigate the fine multifractal spectrum of a class of self-affine Moran sets with fixed frequencies, and we prove that under certain separation conditions, the fine multifractal spectrum $H(\alpha)$ is given by the formula
$$
H(\alpha)=\inf_{-\infty<t<+\infty} \{\alpha t+\beta(t)\}.
$$
\end{abstract}

\maketitle

\section{Introduction}
\subsection{Background}
Let $\mu$ be a Borel regular measure on $\R^d$ with $0<\mu(\R^d)<\infty$ and let $B(x,r)$ be a ball at center $x$ with radius $r$. The \textit{local dimension} of $\mu$ at $x$ is given by
$$
\lcd \mu(x) = \lim_{r\to 0} \frac{\log \mu(B(x,r))}{\log r},
$$
provided the limit exists. For $\alpha\geq 0$ we write
$$
E_\alpha=\{x\in \mathbb{R}^d: \dim_{\mathrm{loc}}\mu(x)=\alpha\}.
$$
The \textit{fine multifractal spectrum or singularity
spectrum} of $\mu$ is defined by
\begin{equation}\label{HHE}
H(\alpha)=\hdd E_\alpha,
\end{equation}
where $\hdd$ denotes the Hausdorff dimension. We refer readers to~\cite{Bk_KJF2,Olsen95} for the background reading.

One main question in multifractal analysis is to investigate the fine multifractal spectrum,  the R\'enyi dimensions and their relations~\cite{AttSel21,Bk_KJF2, Feng12,Olsen95}. There has been an enormous interest in finding the fine multifractal spectra of fractal measures,  such as self-similar measures~\cite{ArbPat96,Falco94,FenLa09,OlsSn11}, self-conformal measures~\cite{Patzsc97}, self-affine measures~\cite{FKJ99,FKJ10, King95,Olsen98,Olsen11}, Gibbs measures~\cite{BJMM07, FKJ99} and Moran measures~\cite{Wumin05, WuXiao11}.

The Bedfor-McMullen carpets~\cite{Bedfo84,McMul84} are a class of simplest self-affine fractals which are often used as a testing ground on questions and conjectures of fractals. In~\cite{King95}, King computed the fine multifractal spectrum for self-affine measures supported on Bedford-McMullen carpets, and he proved that the  fine multifractal formula~\eqref{HHE} holds under certain separation condition which was removed by Jordan and Rams in~\cite{JorRam11}. Olsen generalised King's work onto $\R^d$, and he studied the fine multifractal spectrum for the measures supported on self-affine sponges and Random self-affine sponges, see~\cite{Olsen98,Olsen11}.

Recently, in \cite{GM22, GHM23}, the authors come up with a class of new fractals called self-affine Moran sets which are the generalisation of Bedford-McMullen carpets, and they studied the dimension theory of the sets and the properties of the self-affine Moran measures supported on the sets. It is natural to investigate the multifractal analysis on such sets. In the paper, we study the fine multifractal spectrum for the measures supported self-affine Moran sets. First we review the definitions of self-affine Moran sets and self-affine Moran measures. Then, we state our main conclusions on fine multifractal spectrum for self-affine Moran measures in section~\ref{sec_SCMR}. The proofs  are given in section~\ref{sec_pf}.

\subsection{Self-affine Moran sets and measures}\label{sec_SAMS}
Let  $\{(n_k,m_k)\}_{k=1}^\infty$ be a sequence of integer pairs  such that $n_k\geq 2$ and  $m_k\geq 2$. For each integer $k>0$, let $\mathcal{D}_k$ be a subset of $\{0,\dots,n_k-1\}\times\{0,\dots,m_k-1\}$, and we write $r_k=\card(\D_k)$. We always assume that $r_k\geq 2$.

We write
$$
  \Sigma^{k} = \prod_{j=1}^k\mathcal{D}_j  ,\qquad
  \Sigma^{\infty} = \prod_{j=1}^\infty\mathcal{D}_j, \qquad \Sigma^*=\bigcup_{k=0}^\infty\Sigma^k.
$$
For
$\mathbf{w}=w_1\cdots w_k\in\Sigma^k$,
$\mathbf{v}=v_1\cdots v_l\in\Sigma^l$, write $\mathbf{w}\ast\mathbf{v}=
w_1\cdots w_k v_1\cdots v_l\in\Sigma^{k+l}$. We write
$\mathbf{v}|k = (v_1\cdots v_k)$ for the {\it curtailment}
after $k$ terms of $\mathbf{v} = (v_1 v_2\cdots)\in
\Sigma^{\infty}$. We write $\mathbf{w} \preceq \mathbf{v}$ if $\mathbf{w}$ is a
curtailment of $\mathbf{v}$. We call the set $[\mathbf{w}] =
\{\mathbf{v}\in\Sigma^{\infty} : \mathbf{w} \preceq \mathbf{v}\}$ the {\it cylinder}
of $\mathbf{w}$, where $\mathbf{w}\in \Sigma^*$. If $\mathbf{w}=\emptyset$, its cylinder is $[\mathbf{w}]=\Sigma^{\infty}$.

Given $k>0$, for each $w=(i,j)\in \mathcal{D}_k$, we write $\Phi_k=\operatorname{diag}(n_k^{-1},m_k^{-1})$ for the diagonal matrix, and we define an affine transformation on~$\R^2$ by
\begin{equation}\label{eq:Sk}
  \Psi_w(x)=\Phi_k(x+w), \qquad w\in\mathcal{D}_k.
\end{equation}
For each $\mathbf{w}=(w_1w_2\ldots w_k)\in \Sigma^k$, we write $$
\Psi_{\mathbf{w}}=\Psi_{w_1}\circ \Psi_{w_2}\circ \ldots\circ \Psi_{w_k}.
$$

Suppose that $J=[0,1]^2\subset \mathbb{R}^{2}$. For each integer $k>0$, let $\{\Psi_w\}_{w\in\mathcal{D}_k }$ be the self-affine IFS as in~\eqref{eq:Sk}.
For each $\mathbf{w}\in \Sigma^k$, we write $J_{\mathbf{w}}=\Psi_{\mathbf{w}}(J)$, i.e. $J_{\mathbf{w}}$ is a geometrical affine copy to
$J$. Then we call the non-empty compact set
\begin{equation}\label{attractor}
E=\bigcap\nolimits_{k=1}^{\infty}\bigcup\nolimits_{\mathbf{w}\in \Sigma^{k}} J_{\mathbf{w}}
\end{equation}
a \textit{self-affine Moran set or self-affine Moran carpet}, where the elements $J_{\mathbf{w}}$
are called \textit{\ $k$th-level basic sets} of $E$, see~\cite{GM22, GHM23} for details. 


We denote the projection  $\Pi: \Sigma^\infty \rightarrow \mathbb{R}^2$ by
$$
\Pi(\mathbf{w})= \sum_{k=1}^{\infty} \mathrm{diag} \left(\prod_{h=1}^{k} n_h^{-1}, \prod_{h=1}^{k} m_h^{-1}\right)w_k.
$$
Note that the projection $\Pi$ is surjective. It is clear that $E=\Pi(\Sigma^\infty)$, that is,  the self-affine Moran set is the image of the projection $\Pi$.

For each $\delta>0$, let $k=k(\delta)$ be the unique integer satisfying
\begin{equation}\label{def_k}
 \qquad\frac{1}{m_1}\frac{1}{m_2}\ldots \frac{1}{m_k}\leq \delta<\frac{1}{m_1}\frac{1}{m_2}\ldots \frac{1}{m_{k-1}}.
\end{equation}
If no positive integer satisfies above equation, we always write $k=1$. For each integer $k>0$, let $l=l(k)$ be the unique integer satisfying
\begin{equation} \label{def_l}
\frac{1}{n_1}\frac{1}{n_2}\ldots \frac{1}{n_l}\leq \frac{1}{m_1}\frac{1}{m_2}\ldots \frac{1}{m_{k}}<\frac{1}{n_1}\frac{1}{n_2}\ldots \frac{1}{n_{l-1}}.
\end{equation}
We sometimes write $l(\delta)$ for $l(k)$ if $k=k(\delta)$ is given by~\eqref{def_k}. If there is no ambiguity in the context, we just write $l$ instead of $l(k)$ for simplicity.

For each $\delta>0$ and every $\mathbf{w}=w_1w_2\ldots w_n\ldots \in \Sigma^\infty$, where $w_n=(i_n,j_n)$,  we write
$$
U(\delta, \mathbf{w})=\Big\{\mathbf{v}=v_1v_2\ldots v_n \ldots\in\Sigma^\infty:
\begin{array}{ll}
    i_n=i_n',&n=1,\ldots,l(\delta), \\
    j_n=j_n',&n=1,\ldots, k(\delta),
\end{array}
v_n=(i_n',j_n')\Big\},
$$
and we write $\mathcal{U}_\delta$ for the collection of all such sets, i.e.
$$
\mathcal{U}_\delta=\{U(\delta, \mathbf{w}): w\in \Sigma^\infty\}.
$$
We write
\begin{equation}\label{appsquare}
\mathcal{S}_\delta=\{\Pi(U): U\in \mathcal{U}_\delta\}.
\end{equation}
The elements $S$ of $\mathcal{S}_\delta$ are called the \textit{$\delta$-approximate squares}.  Approximate squares are an essential tool in studying self-affine fractals, see~\cite{Baran07, Bedfo84, LalGa92,McMul84}, and   we may also apply  this tool to explore the fine multifractal spectrum.

To study the multifractal spectrum, we need define self-affine Moran measures on $E$.
For each integer $k>0$, let $\big(p_k(w)>0\big)_{w\in \mathcal{D}_k}$ be a probability vector. For each $\mathbf{w}=w_{1} w_2 \cdots w_{k}\in \Sigma^k$, we write
\begin{equation}\label{mu}
   \widetilde{\mu}([\mathbf{w}])=p_{\mathbf{w}}=p_1(w_{1}) p_2(w_2) \cdots p_k(w_{k}).
\end{equation}
Then $\widetilde{\mu}$ is a Borel measure on $\Sigma^\infty$. It is clear that
\begin{equation}\label{projmu}
  \mu (A)=\widetilde{\mu}(\Pi^{-1}A)
\end{equation}
is a Borel probability measure on $E$, and we call it a \textit{self-affine Moran measure} on $E$. For each $k>0$, we write that, for $w=(i,j)\in \mathcal{D}_k$,
$$
q_k(w)=q_k(j)=\sum_{(i,j)\in \mathcal{D}_k} p_k(i,j),\qquad \widehat{q}_k(w)=\widehat{q}_k(i)=\sum_{(i,j)\in \mathcal{D}_k} p_k(i,j).
$$
For each $S(\delta, x)\in \mathcal{S}_\delta$ where $x\in S(\delta,x)\cap E$, there exists a sequence $\mathbf{w}$ such that $\Pi(\mathbf{w})=x$ and $\Pi(U(\delta, x))= S(\delta, x)$. Then
\begin{equation}\label{muas}
  \mu(S(\delta, x))=\left\{
  \begin{array}{lcl}
    p_1(w_1)\ldots p_l(w_l) q_{l+1}(w_{l+1})\ldots q_k(w_k),  & \  & l \leq k, \\
    p_1(w_1)\ldots p_k(w_k)\widehat{q}_{k+1}(w_{k+1}) \cdots \widehat{q}_l(w_l), & \  & l > k.
  \end{array}
  \right.
\end{equation}
The measure distributed on approximate squares is key to study the fine multifractal spectrum in this paper.

\section{Notation and Main Results}\label{sec_SCMR}
\subsection{Self-affine Moran sets with certain frequency}
In~\cite{GM22,GHM23}, the authors study the dimension theory of self-affine Moran sets and measures with assumption
\begin{equation}\label{UB}
  N^+=\sup\{n_k, m_k: k=1,2,\ldots\}<\infty.
\end{equation}
Since $N^+<\infty$, the patterns in the given sequence $\{(n_k,m_k,\mathcal{D}_k)\}_{k=1}^\infty$ are actually finite, and we write  $\Gamma$ for the collection of  these finite patterns. Therefore,   for all integers $k>0$,
$$
 (n_k,m_k,\mathcal{D}_k)\in \Gamma,
$$
and $\operatorname{card} \Gamma <\infty$. Note that if we impose a Bernoulli measure on $\Gamma^\infty$, by Ergodic theory, for almost all sequences in $\Gamma^\infty$, every element in $\Gamma$ appearing in the sequence has fixed frequency.  Therefore, we investigate the fine multifractal spectrum of self-affine Moran sets with fixed frequencies.

Suppose that  for each pattern $\gamma=(n_\gamma,m_\gamma,\mathcal{D}_\gamma)\in \Gamma$, the limit
\begin{equation}\label{freq}
  \lim_{n\to\infty}\frac{\operatorname{card}\{k: (n_k,m_k,\mathcal{D}_k)=\gamma, k=1,2,\ldots, n \}}{n}
\end{equation}
exists, denoted by $f_\gamma$, and $\sum_{\gamma\in\Gamma} f_\gamma=1$. We call  $E$ the \textit{self-affine Moran set  with frequency $\mathbf{f}=\{f_\gamma\}_{\gamma\in \Gamma}$.}
Let $\mu$ be the self-affine Moran measure given by \eqref{projmu}. It is clear that for the given sequence $\{\big(p_k(w)>0\big)_{w\in \mathcal{D}_k}\}_{k=1}^\infty$ of probability vectors, for each $k>0$, there exists $\gamma\in \Gamma$ such that $\big(p_k(w)\big)_{w\in \mathcal{D}_k}=\big(p_\gamma(w)\big)_{w\in \mathcal{D}_\gamma}$, and the frequency of $\big(p_\gamma(w)\big)_{w\in \mathcal{D}_\gamma}$ appearing in the sequence is also $f_\gamma$. From now on, both notations are used in the context for simplicity.

For the given probability vector $\mathbf{f}=\{f_\gamma\}_{\gamma\in\Gamma}$, we write
\begin{equation}\label{def_zeta}
\zeta:=\frac{\sum_{\gamma\in\Gamma} f_\gamma \log m_\gamma}{\sum_{\gamma\in\Gamma} f_\gamma \log n_\gamma}.
\end{equation}
For each $\gamma\in\Gamma$,  let $\big(p_\gamma(w)>0\big)_{w\in \mathcal{D}_\gamma}$ be a probability vector, and we define $\beta_\gamma(t)$ to be the unique solution to
\begin{eqnarray}\label{begat}
  m_\gamma^{-\beta_\gamma(t)}\sum_{(i,j)\in \D_\gamma}p^t_{\gamma}(ij) u_{\gamma}^{1-\zeta}(j)&=&1,  \qquad \textit{for $\zeta\leq 1$; }  \\
  n_\gamma^{-\beta_\gamma(t)}\sum_{(i,j)\in \D_\gamma}p^t_{\gamma}(ij) \widehat{u}_{\gamma}^{1-\zeta}(i)&=&1,\qquad \textit{for $\zeta > 1$, }  \nonumber
\end{eqnarray}
where $u_{\gamma}(j)=\frac{q_{\gamma}^t(j)}{\sum_{(i,j)\in \D_\gamma} p_{\gamma}^t(i,j)}$
 and {\color{red} $ \widehat{u}_{\gamma}(i)=\frac{\widehat{q}_{\gamma}^t(i)}{\sum_{(i,j)\in \D_\gamma} p_{\gamma}^t(i,j)}$}, and we write
\begin{equation}\label{betat}
  \beta(t)=\left\{ \begin{array}{ll} \frac{\sum_{\gamma\in \Gamma}f_\gamma\beta_\gamma(t) \log m_\gamma}{\sum_{\gamma\in\Gamma}f_\gamma \log m_\gamma} &  \textit{  for }\zeta \leq  1,\\
   \frac{\sum_{\gamma\in \Gamma}f_\gamma\beta_\gamma(t) \log n_\gamma}{\sum_{\gamma\in\Gamma}f_\gamma \log n_\gamma} &  \textit{  for }\zeta> 1.
   \end{array} \right.
\end{equation}

To study the fine multifractal spectrum, the following geometric separation conditions play important roles in the proof , which are also frequently used to study the geometric and topological properties of self-affine sets with grid structures, see~\cite{HM, LMR22}.

We say $E$ satisfies \textit{Row separation condition} (RSC) if for all $\gamma\in \Gamma$ and  for all $(i,j),(i',j')\in D_\gamma$, we have $\Psi_{i,j}(J) \cap \Psi_{i',j'}(J)=\emptyset$ and $|j-j'|\neq 1$.
We say $E $ satisfies \textit{Top or bottom separation condition} (TBSC) if there exists a pattern $\gamma\in \Gamma$ with the frequency $f_\gamma>0$ such that at least one of the following conditions holds:
\begin{itemize}
  \item[(1).]
 For all $(i,j)\in D_\gamma$, $j \neq 0$.
  \item[(2).]
 For all $(i,j)\in D_\gamma$, $j \neq m_\gamma-1$.
\end{itemize}
Note that RSC and TBSC are crucial to the proof of the fine multifractal spectrum for the case $\zeta \leq 1$.  Similarly, we may define column separation condition (CSC) and left or right separation condition (LRSC), which are  crucial for the case $\zeta>1$. More details about these conditions  are discussed in Example~\ref{eg}.   In fact, for  $\zeta>1$, the conclusion follows just by exchanging the roles of $x$ and $y$ axes. Therefore  we always assume that $\zeta \leq 1$ throughout the paper and omit the proof for $\zeta >1$.

Let
\begin{eqnarray*}
\alpha_{\min}&=&\frac{\sum_{\gamma\in\Gamma} f_\gamma \min_{w\in \D_\gamma}\big\{-\zeta\log p_{\gamma}(w) - (1-\zeta)\log q_{\gamma}(w)\big\}}{\sum_{\gamma\in\Gamma}f_\gamma \log m_\gamma},
\\
\alpha_{\max}&=&\frac{\sum_{\gamma\in\Gamma} f_\gamma \max_{w\in \D_\gamma} \big\{-\zeta\log p_{\gamma}(w) - (1-\zeta)\log q_{\gamma}(w)\big\}}{\sum_{\gamma\in\Gamma}f_\gamma \log m_\gamma}.
\end{eqnarray*}
The interval $(\alpha_{\mathrm{min}},\alpha_{\mathrm{max}})$ gives the proper domain that $H(\alpha)$ is valid.

\subsection{Main conclusions}

We have the following conclusions for self-affine Moran measures.
\begin{thm}\label{thm_mfa}
Let $E$ be the self-affine Moran set satisfying either RSC or TBSC with the frequency $\mathbf{f}=\{f_\gamma\}_{\gamma\in\Gamma}$ and $\zeta\leq 1$. Let $\mu$ be the self-affine Moran measure given by \eqref{projmu} with $\big(p_\gamma(w)>0\big)_{w\in \mathcal{D}_\gamma}$. Then for every $\alpha\in (\alpha_{\mathrm{min}},\alpha_{\mathrm{max}})$, we have that
$$
H(\alpha)=\inf_t \{\alpha t+\beta(t)\}.
$$
Furthermore, $H(\alpha)$ is differentiable with respect to $\alpha$ and is concave.
\end{thm}
Note that (1) For $\zeta >1$, by replacing RSC and TBSC by CSC and LRSC, the fine multifractal spectrum formula is  still valid.  (2)  RSC and TBSC are  required for the proof of upper bound.  Actually, we prove the upper bound holds in a more general case, see Theorem~\ref{thmub} , where RSC and TBSC are replaced by the replica condition  in Section~\ref{sec_ub}. It would be very interesting if one may prove the conclusion without the replica condition.

For each $\gamma\in \Gamma$, we write that
\begin{eqnarray*}
  r_\gamma(j)&=&\operatorname{card}\{i\colon (i,j)\in \mathcal{D}_\gamma \textrm{ for each } j\}, \\
  \widehat{r}_\gamma (i)&=&\operatorname{card}\{j\colon (i,j)\in \mathcal{D}_\gamma \textrm{ for each } i\},
\end{eqnarray*}
Instead of using these geometric separation conditions, we may alternatively make some assumptions on $r_\gamma$ and $p_\gamma$, and the fine multifractal formula  is still valid.

\begin{thm}\label{cor2}
Let $E$ be the self-affine Moran set with the frequency $\mathbf{f}=\{f_\gamma\}_{\gamma\in\Gamma}$ and $\zeta\leq 1$. Let $\mu$ be the self-affine Moran measure given by \eqref{projmu} with $\big(p_\gamma(w)>0\big)_{w\in \mathcal{D}_\gamma}$.  Suppose that, for all $\gamma\neq \gamma'\in \Gamma$,
\begin{itemize}
  \item[(1).]
 $r_{\gamma}(0)=r_{\gamma'}(0)$, and $r_{\gamma}(m_{\gamma}-1)=r_{\gamma'}(m_{\gamma'}-1);$
  \item[(2).]
 $\{p_{\gamma'}(i,0)\}_{(i,0)\in \mathcal{D}_{\gamma'}}$ and $\{p_{\gamma'}(i,m_{\gamma'}-1)\}_{(i,m_{\gamma'}-1)\in \mathcal{D}_{\gamma'}}$ are  the permutations of $\{p_{\gamma}(i,0)\}_{(i,0)\in \mathcal{D}_{\gamma}}$ and $\{p_{\gamma}(i,m_{\gamma}-1)\}_{(i,m_{\gamma}-1)\in \mathcal{D}_{\gamma}}$, respectively.
\end{itemize}
Then for every $\alpha\in (\alpha_{\mathrm{min}},\alpha_{\mathrm{max}})$, we have that
$$
H(\alpha)=\inf_t \{\alpha t+\beta(t)\}.
$$
\end{thm}

The fine multifractal spectrum of Bedford-McMullen set is a direct consequence of Theorem~\ref{cor2}, which was first studied by King~\cite{King95}, and improved by Jordan and Rams~\cite{ JorRam11} .
\begin{cor}
Let $\mu$ be a self-affine measure supported on a Bedford-McMullen set $E$. Then for every  $\alpha\in (\alpha_{\mathrm{min}},\alpha_{\mathrm{max}})$, we have that
$$
H(\alpha)=\inf_t \{\alpha t+\beta(t)\}.
$$
\end{cor}

Finally, we give an example to illustrate the separations conditions.
\begin{exmp}\label{eg}
\begin{figure}[h]
\centering
\includegraphics[width=\textwidth]{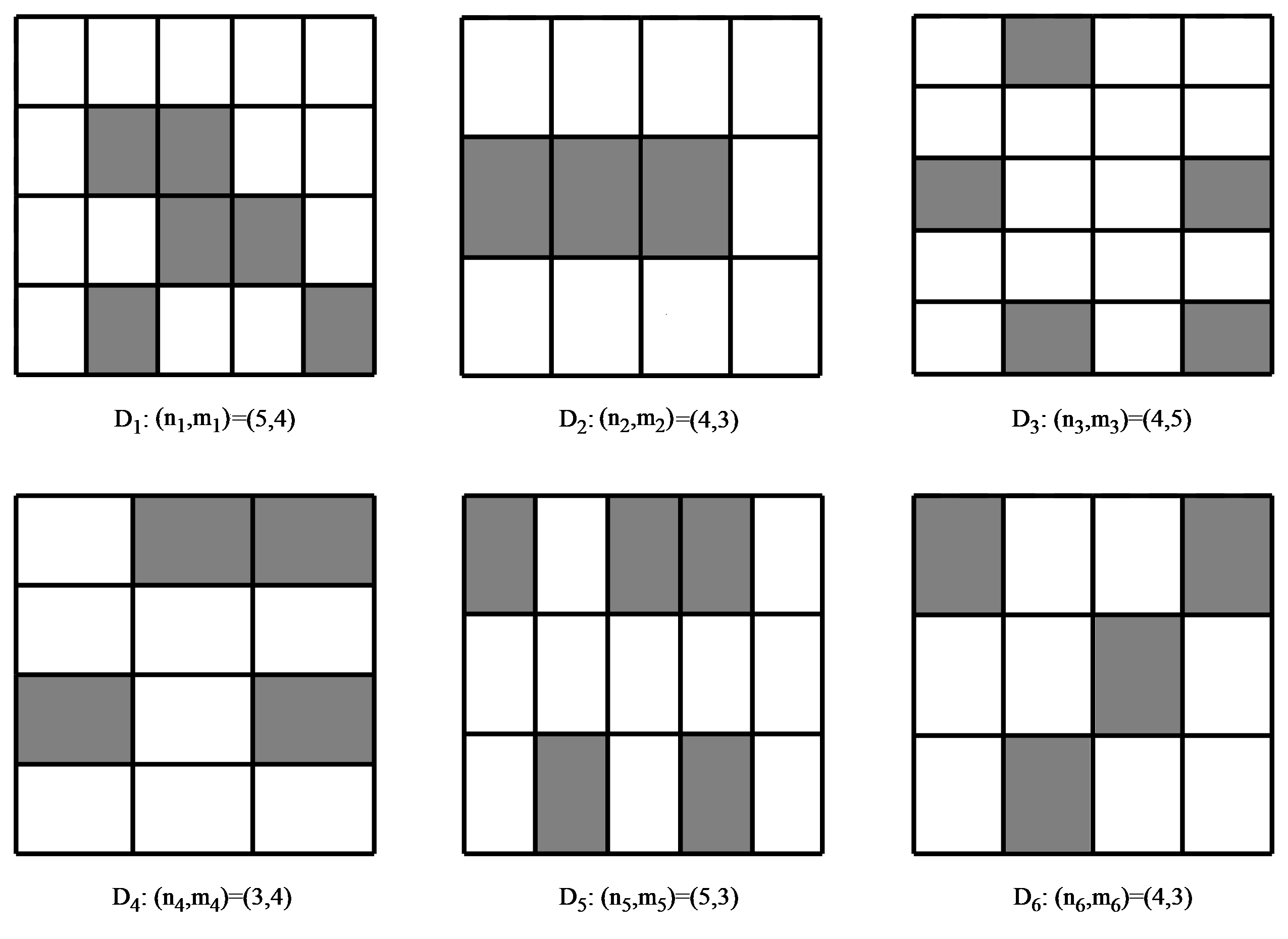}
\caption{}\label{fig_cdt}
\end{figure}
Let $(n_1,m_1)=(5,4)$, $(n_2,m_2)=(4,3)$, $(n_3,m_3)=(4,5)$, $(n_4,m_4)=(3,4)$, $(n_5,m_5)=(5,3)$, {\color{red} $(n_6,m_6)=(4,3)$} and
\begin{eqnarray*}
\D_1&=&\{(1,0),(1,2),(2,1),(2,2),(3,1),(4,0)\},  \quad \
\D_2=\{(0,1),(1,1),(2,1)\}, \\
\D_3&=&\{(0,2),(1,0),(1,4),(3,0),(3,2)\}, \qquad\qquad
\D_4=\{(0,1),(1,3),(2,1),(2,3)\}, \\
\D_5&=&\{(0,2),(1,0),(2,2),(3,0),(3,2)\},\qquad\qquad
\D_6=\{(0,2),(1,0),(2,1),(3,2)\},
\end{eqnarray*}
see Figure~\ref{fig_cdt}.

(1). Let $\Gamma=\{\D_1,\D_3,\D_6\}$,  and   $f_1=\frac{1}{4},f_2=\frac{1}{2},f_4=\frac{1}{4}$.  Then by \eqref{def_zeta},
$$
\zeta= \frac{0.25\log 4+0.5\log 5 + 0.25\log 3}{0.25\log 5+0.5\log 4 +0.25\log 4}\approx 0.9888<1.
$$
Since the top row of $\D_1$ is empty and $f_1>0$, $E$ satisfies top or bottom separation condition.

If we set $f_1=\frac{1}{10},f_2=\frac{4}{5},f_4=\frac{1}{10}$, we have {\color{red} $\zeta\approx 1.090>1$}, and $E$ satisfies left or right separation condition.

(2) Suppose that $\Gamma=\{\D_2,\D_3,\D_5\}$,  and   $f_1=\frac{1}{4},f_2=\frac{1}{2},f_4=\frac{1}{4}$.  Then $\zeta\approx 0.9767<1,$
and  $E$ satisfies both RSC and TBSC.

(3) Suppose that $\Gamma=\{\D_3,\D_4,\D_5\}$,  and $f_3=\frac{1}{5},f_4=\frac{1}{5},f_5=\frac{3}{5}$. Then $\zeta\approx 0.860 <1,$
and $E$ satisfies Row separation condition.  We also compute its fine multifractal spectrum as a simple example to illustrate our main conclusion.  Let $\beta_3(t)$, $\beta_4(t)$ and $\beta_5(t)$ be the solutions to
\begin{eqnarray*}
  5^{-\beta_3(t)}\sum_{(i,j)\in \D_3}p^t_3(ij) u_3^{1-\zeta}(j)=1, && \qquad 4^{-\beta_4(t)}\sum_{(i,j)\in \D_4}p^t_4(ij) u_4^{1-\zeta}(j)=1, \\
  3^{-\beta_5(t)}\sum_{(i,j)\in \D_5}p^t_5(ij) u_5^{1-\zeta}(j)=1. &&
\end{eqnarray*}
By \eqref{betat}, we have that
$$
\beta(t)=\frac{\beta_3(t)\log 5+2\beta_4(t)\log 2+3\beta_5(t)\log 3}{2\log 2+3\log 3+\log 5}
$$
By Theorem~\ref{thm_mfa}, for every $\alpha\in (\alpha_{\mathrm{min}},\alpha_{\mathrm{max}})$, the fine multifractal spectrum is given by $H(\alpha)=\inf_t \{\alpha t+\beta(t)\}.$

\end{exmp}

\section{Multifractal spectrum of self-affine Moran measures}\label{sec_pf}
In this section, we investigate the fine multifractal spectrum of self-affine Moran measures. The proof is divided into two parts, one is to show that $ \inf_t\{t\alpha+\beta(t)\}$ is the lower bound, the other is to show that it is the upper bound.

To study the lower bound of the multifractal spectrum of self-affine Moran measures, we need the following two conclusions. The first is a version of law of large numbers, and we refer the readers to ~\cite[Theorem 1 in section 5.2]{CYS78} for details.
\begin{thm}\label{lln}
Let $\{X_n\}_{n=1}^\infty $ be a sequence of independent random variables satisfying that
$$
\sum_{n=1}^\infty \frac{\mathbb{E} (|X_n|^{\alpha_n} )}{n^{\alpha_n}}<\infty
$$
for some choice of $\alpha_n$ in $(0,2],$ where $\mathbb{E} (X_n)=0$ whenever $1\leq \alpha_n\leq 2$.
Then the sequence $\frac{1}{n}\sum_{i=1}^n X_i$ converges to 0 almost surely.
\end{thm}

The next  is a version of Frostman Lemma, which is useful in finding the dimension of measures, see~\cite{Falco97} for details.
\begin{thm}\label{lem_frost}
Let $\nu$ be a finite Borel measure on $\R^d$. If $\liminf_{r\to 0} \frac{\log \nu(B(x,r)}{\log r} =s$ for $\nu$-almost every $x$. Then $\hdd \nu= s.$\\
\end{thm}

For $\delta=\frac{1}{m_1 m_2\ldots m_k} $, we simply write that $\mathcal{S}_k= \mathcal{S}_\delta$. For each $\mathbf{w}\in \Sigma^\infty$, there exists $U(\delta, \mathbf{w})\in \mathcal{U}_\delta$, and we write $S_k(\mathbf{w})$ for the approximate square in $\mathcal{S}_k$ containing $\Pi(w)$, i.e.
$$
S_k(\mathbf{w})=\Pi(U(\delta, \mathbf{w}))\in\mathcal{S}_k.
$$
Recall that, for the sequence $\big\{\big(p_k(w)>0\big)_{w\in \mathcal{D}_k}\big\}_{k=1}^\infty$ of probability vectors and  $w=(i,j)\in\D_k$,
$$
q_k(w)=q_k(j)=\sum_{(i,j)\in \mathcal{D}_k} p_k(i,j),\qquad \widehat{q}_k(w)=\widehat{q}_k(i)=\sum_{(i,j)\in \mathcal{D}_k} p_k(i,j)
$$
and
\begin{equation}\label{eq_uk}
u_k(w)=u_k(j)=\frac{q_k^t(j)}{\sum_{(i,j)\in \D_k} p_k^t(i,j)}.
\end{equation}

\subsection{Lower Bound}
We need to define a new measure to study the lower bound. For each $k>0$, given a real  $t>0$, we define $\beta_k(t)$ as the solution to
$$
  m_k^{-\beta_k(t)}\sum_{(i,j)\in \D_k}p^t_k(ij) u_k^{1-\zeta}(j)=1.
$$
We write
$$
  P_k(w)=P_k(i,j)=m_k^{- \beta_k(t)}p^t_k(i,j) u_k^{1-\zeta}(j),
$$
and
$$
  Q_k(w)=Q_k(j)=\sum_{(i,j)\in \D_k} P_k(i,j), \qquad \widehat{Q}_k(w)=\widehat{Q}_k(i)=\sum_{(i,j)\in \D_k \atop} P_k(i,j).
$$
It is clear that $(P_k(w))_{w\in \mathcal{D}_k}$, $(Q_k(w))_{w\in \mathcal{D}_k}$ and $(\widehat{Q}_k(w)_{w\in \mathcal{D}_k}$ are probability vectors.

Let $\widetilde{\mu}_t$ be the product Borel probability on $\Sigma^\infty$ defined by
\begin{equation}\label{defmut}
  \widetilde{\mu}_t([w_1w_2 \ldots w_k])=P_1(w_1)P_2(w_2)\ldots P_k(w_k),
\end{equation}
 and $\mu_t$ be the image measure of $\widetilde{\mu}_t$ under $\Pi$, i.e.
\begin{equation}\label{muttil}
  \mu_t=\widetilde{\mu}_t \circ \Pi^{-1}.
\end{equation}

Similar to \eqref{muas}, for each $S_k(\mathbf{w})\in \mathcal{S}_k$, the $\mu_t$ measure distributed on $S_k(\mathbf{w})$ is given by 
\begin{equation}\label{mutas}
  \mu_t(S_k(\mathbf{w}))=\left\{
  \begin{array}{lcl}
    P_1(w_1)\ldots P_l(w_l) Q_{l+1}(w_{l+1})\ldots Q_k(w_k),  & \  & l \leq k, \\
    P_1(w_1)\ldots P_k(w_k) \widehat{Q}_{k+1}(w_{k+1})\ldots \widehat{Q}_l(w_l), & \  & l > k.
  \end{array}
  \right.
\end{equation}

Recall that $E$ is the self-affine Moran set with frequency $\{f_\gamma\}_{\gamma\in \Gamma}$, we write
\begin{equation}\label{def_alpha}
\alpha(t)=\frac{\sum_{\gamma\in\Gamma} f_\gamma \sum_{w\in\D_\gamma}P_{\gamma}(w)\Big(-\zeta\log p_{\gamma}(w) - (1-\zeta)\log q_{\gamma}(w)\Big)}{\sum_{\gamma\in\Gamma}f_\gamma \log m_\gamma}.
\end{equation}

\begin{lem}\label{lb}
Let $E$ be the self-affine Moran set with the frequency $\{f_\gamma\}_{\gamma\in \Gamma}$.  Let $\mu$ and $\widetilde{\mu}_t$ be the measures given by \eqref{projmu} and~\eqref{defmut} respectively.
Then for $\widetilde{\mu}_t$-almost all $\mathbf{w}$ we have that
$$
\lim_{k\to \infty}\frac{\log \mu(S_k(\mathbf{w}))}{-\log m_1 m_2\ldots m_k}=\alpha(t).
$$
\end{lem}

\begin{proof}
Let $\{X_i\}_{i=1}^\infty$ be a sequence of  independent random variables on $(\Sigma^\infty,\mathcal{B},\mu_t)$ given by
$$
X_i(\mathbf{w})=\log p_i(w_i) - \sum_{w'\in\D_i} P_i(w') \log p_i(w'), \quad   \mathbf{w}=w_1w_2 \ldots w_i \ldots \in \Sigma^\infty.
$$
Note that there are only finitely many patterns in $\Gamma$, by simple calculation, we obtain that
$$
\mathbb{E}(X_i)=0 \qquad \text{and} \qquad \sum_{i=1}^{\infty}\frac{\mathbb{E}(X_i^2)}{i^2}<\infty.
$$
Hence, by Theorem~\ref{lln}, we have that
\begin{equation}\label{EX}
\lim_{k\to \infty} \frac{1}{k} \sum_{i=1}^k X_i(\mathbf{w})=0,
\end{equation}
for $\widetilde{\mu}_t$-a.e. $\mathbf{w}\in \Sigma^\infty$.

Similarly, by Theorem~\ref{lln}, for the random variables
$$
Y_i(\mathbf{w})=\log q_i(w_i) - \sum_{j'=0}^{m_i-1} Q_i(j') \log q_i(j'), \qquad w_i=(i',j')\in \mathcal{D}_i,
$$
we have that, for $\widetilde{\mu}_t$-a.e. $\mathbf{w}\in \Sigma^\infty$,
\begin{equation}\label{EY}
\lim_{k\to \infty} \frac{1}{k} \sum_{i=1}^k Y_i(\mathbf{w})=0.
\end{equation}
Therefore, by~\eqref{muas}, for each integer $k>0$, we obtain that
$$
\log \mu(S_k(\mathbf{w}))=\left\{\begin{array}{l}\sum_{i=1}^{l} X_i(\mathbf{w}) +
 \sum_{i=1}^{l} \sum_{w'\in \D_{i}} P_{i}(w') \log p_{i}(w')   \\ \hspace{0.5cm}+ \sum_{i=l+1}^{k} Y_i(\mathbf{w}) + \sum_{i=l+1}^{k} \sum_{w'\in \D_{i}} P_{i}(w') \log q_{i}(w') , \qquad  l\leq k\\
\sum_{i=1}^{k} X_i(\mathbf{w}) +  \sum_{i=1}^{k} \sum_{w'\in \D_{i}} P_{i}(w') \log p_{i}(w') \\ \hspace{0.5cm}+\sum_{i=k+1}^{l} Y_i(\mathbf{w}) + \sum_{i=k+1}^{l} \sum_{w'\in \D_{i}} P_{i}(w') \log \widehat{q}_{i}(w'), \qquad l>k,
\end{array} \right.
$$
where $l=l(k)$ is given by~\eqref{def_l}.

For each $\gamma\in \Gamma$, we write
$$c_\gamma(n)=\operatorname{card}\{k: (n_k,m_k,\mathcal{D}_k)=\gamma, k=1,2,\ldots, n \},
$$
and it is clear that
$f_\gamma=\lim_{n\to \infty} \frac{c_\gamma(n)}{n}$.  It is obvious that
\begin{equation}\label{denmk}
\lim_{k\to\infty } \frac{\log m_1\ldots m_k}{k}=\sum_{\gamma\in \Gamma} f_\gamma \log m_\gamma.
\end{equation}
By~\eqref{def_l} and~\eqref{def_zeta}, we have that
\begin{equation}\label{lim_kl}
\lim_{k\to\infty } \frac{l}{k}=\zeta.
\end{equation}

For $\zeta<1$,  there exists $K>0$ such that for all $k>K$ we have that $l<k$. Combining with~\eqref{EX} and~\eqref{EY}, we have that
\begin{eqnarray*}
 &&\hspace{-0.8cm} \lim_{k\to \infty} \frac{\log \mu(S_k(\mathbf{w}))}{k}
  = \lim_{k\to \infty} \left(\frac{1}{k}\sum_{i=1}^{l} \sum_{w'\in \D_{i}} P_{i}(w') \log p_{i}(w') + \frac{1}{k}\sum_{i=l+1}^{k} \sum_{w'\in \D_{i}} P_{i}(w') \log q_{i}(w')\right) \\
  &&= \lim_{k\to \infty} \sum_{\gamma\in\Gamma}\left(\frac{c_\gamma(l)}{k}  \sum_{w\in\D_\gamma}P_{\gamma}(w)\log p_{\gamma}(w) + \frac{c_\gamma(k)-c_\gamma(l)}{k} \sum_{w\in\D_\gamma}P_{\gamma}(w)\log q_{\gamma}(w)\right)\\
  &&= \sum_{\gamma\in\Gamma} f_\gamma\sum_{w\in\D_\gamma}P_{\gamma}(w)\Big(\zeta\log p_{\gamma}(w) + (1-\zeta)\log q_{\gamma}(w)\Big),
\end{eqnarray*}
for $\widetilde{\mu}_t$-a.e. $\mathbf{w}\in \Sigma^\infty$. Combine with \eqref{denmk}, we have that
$$
\lim_{k\to \infty}\frac{\log \mu(S_k(\mathbf{w}))}{-\log m_1\ldots m_k}
=\frac{\sum_{\gamma\in\Gamma} f_\gamma\sum_{w\in\D_\gamma}P_{\gamma}(w)\Big(-\zeta\log p_{\gamma}(w) - (1-\zeta)\log q_{\gamma}(w)\Big)}{\sum_{\gamma\in\Gamma}f_\gamma \log m_\gamma}.
$$

For $\zeta=1$, by\eqref{lim_kl},  we have that $\lim_{k\to\infty } \frac{l}{k}=1$. Therefore, both $k\leq l$ and $l>k$ may appear alternately, and we only  estimate the following for $l>k$ (the other case is identical).
\begin{eqnarray*}
&&\left|\sum_{\gamma\in\Gamma}f_\gamma \sum_{w\in\D_\gamma}P_{\gamma}(w)\log p_{\gamma}(w) - \frac{\log \mu(S_k(\mathbf{w}))}{k}\right| \\
&&=\bigg|\sum_{\gamma\in\Gamma}f_\gamma \sum_{w\in\D_\gamma}P_{\gamma}(w)\log p_{\gamma}(w)- \frac{1}{k} \sum_{i=1}^k X_i(\mathbf{w})- \frac{1}{k}\sum_{i=1}^{k} \sum_{w'\in \D_{i}} P_{i}(w') \log p_{i}(w') \\
&&\hspace{6.5cm}  - \frac{1}{k} \sum_{i=k+1}^l Y_i(\mathbf{w}) - \frac{1}{k} \sum_{i=k+1}^{l} \sum_{i'=0}^{n_i} \widehat{Q}_{i}(i') \log \widehat{q}_{i}(i') \bigg|  \\
&&\leq  \bigg| \frac{1}{k} \sum_{i=1}^k X_i(\mathbf{w}) + \frac{1}{k} \sum_{i=1}^k Y_i(\mathbf{w}) \bigg|  +\bigg|\sum_{\gamma\in\Gamma}f_\gamma \sum_{w\in\D_\gamma} P_{\gamma}(w) \log p_{\gamma}(w)  \\
&&\hspace{2cm}  -\sum_{\gamma\in\Gamma} \bigg(\frac{c_\gamma(l)}{k}  \sum_{w\in\D_\gamma}P_{\gamma}(w)\log p_{\gamma}(w) + \frac{c_\gamma(l)-c_\gamma(k)}{k} \sum_{i'=0}^{n_\gamma}  \widehat{Q}_{i}(i') \log \widehat{q}_{i}\bigg)\bigg|.
\end{eqnarray*}
Combining~\eqref{EX}, \eqref{EY} and  \eqref{lim_kl} with $f_\gamma=\lim_{k\to \infty} \frac{c_\gamma(k)}{k}$ and $\lim_{k\to\infty } \frac{l}{k}=1$, we have that
\begin{eqnarray*}
\hspace{-0.8cm} \lim_{k\to \infty} \frac{\log \mu(S_k(\mathbf{w}))}{k}
  &=& \sum_{\gamma\in\Gamma}f_\gamma \sum_{w\in\D_\gamma}P_{\gamma}(w)\log p_{\gamma}(w),
\end{eqnarray*}
for $\widetilde{\mu}_t$-a.e. $\mathbf{w}\in \Sigma^\infty$.
 Combine with \eqref{denmk}, we have that
$$
\lim_{k\to \infty}\frac{\log \mu(S_k(\mathbf{w}))}{-\log m_1\ldots m_k}
=\frac{-\sum_{\gamma\in\Gamma} f_\gamma\sum_{w\in\D_\gamma}P_{\gamma}(w) \log p_{\gamma}(w) }{\sum_{\gamma\in\Gamma}f_\gamma \log m_\gamma}.
$$

Therefore for $\zeta\leq 1$, by \eqref{def_alpha}, it follows that
\begin{eqnarray*}
\lim_{k\to \infty}\frac{\log \mu(S_k(\mathbf{w}))}{-\log m_1\ldots m_k}
&=&\alpha(t),
\end{eqnarray*}
and the conclusion holds.
\end{proof}

\begin{lem}\label{lem_ld}
Let $E$ be the self-affine Moran set with the frequency $\mathbf{f}=\{f_\gamma\}_{\gamma\in\Gamma}$ and $\zeta\leq 1$. Let $\mu$ and $\mu_t$ be given by \eqref{projmu} and~\eqref{muttil}. Then, for $\mu_t$-almost all $x$,
$$
    \lcd \mu(x) =\alpha(t).
$$
\end{lem}
\begin{proof}

For each integer $k>0$, we write
$$
A_k=\{x\in E: B\big(x,(m_1\ldots m_k)^{-1}e^{-\sqrt{k}}\big)\cap E \subset S_k(\mathbf{w}), \text{ for some }\bw\in\Pi^{-1}(\{x\})\}.
$$

For each $x\in E$, choose $\bw=(i_1,j_1)(i_2,j_2)\ldots\in\Pi^{-1}(\{x\})$. Let $\widetilde{S}_k(\bw)$ be the rectangle with height $(m_1 \ldots m_k)^{-1}$ and width $(n_1 \ldots n_l)^{-1}$ given by 
$$
\widetilde{S}_k(\bw)=\Big[\sum_{h=1}^l\frac{i_h}{n_1 \ldots n_h}, \sum_{h=1}^l\frac{i_h}{n_1 \ldots n_h}+\frac{1}{n_1 \ldots n_l}\Big]\times \Big[\sum_{h=1}^k\frac{j_h}{m_1 \ldots m_h}, \sum_{h=1}^k\frac{j_h}{m_1 \ldots m_h}+\frac{1}{m_1 \ldots m_k}\Big].
$$

 It is clear that $\widetilde{S}_k(\bw) \cap E = S_k(\bw)$.
Let $L_k^B$($L_k^T,L_k^L,L_k^R$) be the collection of $x\in E$ such that the distance from $x$ to the bottom(top, left, right) side of $\widetilde{S}_k(\bw)$ is less than $(m_1\ldots m_k)^{-1}e^{-\sqrt{k}}$, for some $\bw\in\Pi^{-1}(\{x\})$. It is clear that $$A_k^c \subset L_k^B \cup L_k^T \cup L_k^L \cup L_k^R.$$

For each $x\in L_k^B$, it is clear that $j_{k+1} = \ldots =j_{k+[\sqrt{k}/\log N^+]}=0$, and  the measure of $L_k^B$ is bounded by
$$
\mu_t(L_k^B) \leq Q_{k+1}(0)\ldots Q_{k+[\sqrt{k}/\log N^+]}(0).
$$
By the same argument, we have that
\begin{eqnarray*}
\mu_t(L_k^T) &\leq& Q_{k+1}(m_{k+1}-1)\ldots Q_{k+[\sqrt{k}/\log N^+]}(m_{k+[\sqrt{k}/\log N^+]}-1);\\
\mu_t(L_k^L) &\leq& \widehat{Q}_{l+1}(0)\ldots \widehat{Q}_{l+[\sqrt{l}/\log N^+]}(0);\\
\mu_t(L_k^R) &\leq& \widehat{Q}_{l+1}(n_{l+1}-1)\ldots Q_{l+[\sqrt{l}/\log N^+]}(n_{l+[\sqrt{l }/\log N^+]}-1).
\end{eqnarray*}

Let $Q^+=\max_{\gamma\in\Gamma} \{Q_{\gamma}(0)<1,Q_{\gamma}(m_\gamma-1)<1,\widehat{Q}_\gamma(0)<1, \widehat{Q}_\gamma(n_\gamma-1) <1\}^{1/(\log N^+ +1)}$. Since $  \operatorname{card}<\infty$, it is clear that $Q^+<1$.
If $Q_{k'}(0)\neq 1$ for all $k+1\leq k'\leq k+[\sqrt{k}/\log N^+]$,  the measure of $L_k^B$ is bounded by
$$
\mu_t(L_k^B) \leq Q_{k+1}(0)\ldots Q_{k+[\sqrt{k}/\log N^+]}(0) \leq (Q^+)^{\sqrt{k}}.
$$
Otherwise, there exists an integer $k_0$, $k+1\leq k_0 \leq k+[\sqrt{k}/\log N^+]$ such that $Q_{k_0}(0)=1$. It implies that $r_{k_0}(m_{k_0}-1)=0$, that is,
 $L_k^T$ is empty, and for each $x\in L_k^B$, there exists $\bw\in \Pi^{-1}(\{x\})$ such that
\begin{eqnarray*}
B\big(x,(m_1\ldots m_k)^{-1}e^{-\sqrt{k}}\big)\cap E &\subset& B\big(x,(m_1\ldots m_{k_0})\big)\cap E \\
&\subset& \bigcup\{S_{k}(\mathbf{w}')\colon j_1'=j_1, \ldots, j_{k}'=j_{k}\}.
\end{eqnarray*}
Hence, for each $x \in L_k^B \backslash (L_k^L\cup L_k^R)$, we have that $B\big(x,(m_1\ldots m_k)^{-1}e^{-\sqrt{k}}\big)\cap E \subset S_k(\bw)$ for some  $\bw\in \Pi^{-1}(\{x\})$, that is to say,  $x\in A_k$. Therefore, we obtain that $A_k^c \subset L_k^L \cup L_k^R$.

The similar argument applies to each of the other three sides.

For each given $k>0$, we have to estimate the measure distributed on  $A_k^c$ in the following four cases.

(1)For all $k+1<k'<k+[\sqrt{k}/\log N^+]$ and all $l+1 \leq l'\leq l+[\sqrt{l}/\log N^+]$, we have that $Q_{k'}(0),Q_{k'}(m_{k'}-1), \widehat{Q}_{l'}(0), \widehat{Q}_{l'}(n_{l'}-1) < 1$. Since $\mu_t(L_k^*)\leq (Q^+)^{\sqrt{k}}$ where $*= B, T, L, R$, it is clear that
$$
\mu_t(A_k^c) \leq \mu_t(L_k^B) + \mu_t(L_k^T) + \mu_t(L_k^L) + \mu_t(L_k^R) \leq 4(Q^+)^{\sqrt{k}}.
$$

(2)There exists an integer $k+1\leq k_0 \leq k+[\sqrt{k}/\log N^+]$ such that $Q_{k_0}(0)=1$ or $Q_{k_0}(m_{k_0}-1)=1$, and for all $l+1 \leq l'\leq l+[\sqrt{l}/\log N^+]$, $\widehat{Q}_{l'}(0), \widehat{Q}_{l'}(n_{l'}-1) <1$. Since $A_k^c \subset L_k^L \cup L_k^R$ and $\mu_t(L_k^*)\leq (Q^+)^{\sqrt{k}}$ where $*=L, R$,  it is clear that
$$
\mu_t(A_k^c) \leq  \mu_t(L_k^L) + \mu_t(L_k^R) \leq 2(Q^+)^{\sqrt{k}}.
$$

(3)There exists an integer $l+1\leq l_0 \leq l+[\sqrt{l}/\log N^+]$ such that $\widehat{Q}_{l_0}(0)=1$ or $\widehat{Q}_{l_0}(n_{l_0}-1)=1$, and for all $k+1 \leq k'\leq k+[\sqrt{k}/\log N^+]$, $Q_{k'}(0), Q_{k'}(n_{k'}-1) < 1$. Similar to (2), we have that
$$
\mu_t(A_k^c) \leq  \mu_t(L_k^B) + \mu_t(L_k^T) \leq 2(Q^+)^{\sqrt{k}}.
$$

(4)There exists $k+1\leq k_0 \leq k+[\sqrt{k}/\log N^+]$ such that $Q_{k_0}(0)=1$ or $Q_{k_0}(m_{k_0}-1)=1$, and $l+1\leq l_0 \leq l+[\sqrt{l}/\log N^+]$ such that $Q_{l_0}(0)=1$ or $\widehat{Q}_{l_0}(n_{l_0}-1)=1$. At least one of $L_k^B$ and $L_k^T$ is empty, and one of $L_k^B$ and $L_k^T$ is empty. This implies that for each $x\in E$, there exists $\bw\in \Pi^{-1}(\{x\})$ such that $B\big(x,(m_1\ldots m_k)^{-1}e^{-\sqrt{k}}\big)\cap E \subset S_k(\bw)$, and  it immediately follows that $E \subset A_k$. Hence $A_k^c \cap E=\emptyset$, and thus $\mu_t(A_k^c)=0$.

Therefore the measure of $A_k^c$ is bounded by $\mu_t(A_k^c)\leq 4(Q^+)^{\sqrt{k}}$. Since $Q^+<1$, it follows that
$$
\sum_{k=1}^\infty \mu_t(A_k^c)  \leq 4 \sum_{k=1}^\infty (Q^+)^{\sqrt{k}}<\infty.
$$
By Borel-Cantelli Lemma, it follows that
$$\mu_t(A_k^c\,i.o.)=0.
$$
Therefore, for $\mu_t$-almost all $x$, we have that
$$
\mu\Big(B\big(x, (m_1\ldots m_k)^{-1}e^{-\sqrt{k}}\big)\Big)\leq \mu(S_k(\mathbf{w})),
$$
for sufficiently large $k$.  For each $r>0$, there exists a unique integer $k$ such that
$$
(m_1\ldots m_{k+1})^{-1}e^{-\sqrt{k+1}}\leq r<(m_1\ldots m_k)^{-1}e^{-\sqrt{k}},
$$
which implies that $\mu(B(x,r))\leq \mu(S_k(x))$. Therefore, we have that
$$
\liminf_{r\to 0}\frac{\log \mu(B(x,r))}{\log r}\geq \liminf_{k\to \infty}\frac{\log \mu(S_k(\mathbf{w}))}{-\log m_1 m_2\ldots m_k}.
$$

Similarly, we have that
$$
\limsup_{r\to 0}\frac{\log \mu(B(x,r))}{\log r}\leq \limsup_{k\to \infty}\frac{\log \mu(S_k(\mathbf{w}))}{-\log m_1 m_2\ldots m_k}.
$$
By Lemma~\ref{lb}, The conclusion holds.
\end{proof}

\begin{lem}\label{dimmut}
Let $E$ be the self-affine Moran set with the frequency $\mathbf{f}=\{f_\gamma\}_{\gamma\in\Gamma}$ and $\zeta\leq 1$.
Let $\widetilde{\mu}_t$ and $\mu_t$ be defined by \eqref{defmut} and \eqref{muttil}. If
$$
\liminf_{k\to \infty}\frac{\log \mu_t(S_k(\mathbf{w}))}{-\log m_1 m_2\ldots m_k}=\delta,
$$
for $\widetilde{\mu}_t$-a.e. $\mathbf{w}\in \Sigma^\infty$, then $\hdd \mu_t=\delta$.
\end{lem}

\begin{proof}
Using the same argument as in Lemma~\ref{lem_ld}, we have that
$$
\lim_{k\to \infty}\frac{\log \mu_t(B(x,r))}{\log r} = \lim_{k\to \infty}\frac{\log \mu_t(S_k(\mathbf{w}))}{-\log m_1 m_2\ldots m_k},
$$
for $\mu_t$-a.e. $x\in E$. By Theorem~\ref{lem_frost}, we conclude that $\hdd \mu_t=\delta$.
\end{proof}

\begin{lem}\label{mut}
Let $E$ be the self-affine Moran set with the frequency $\mathbf{f}=\{f_\gamma\}_{\gamma\in\Gamma}$ and $\zeta\leq 1$.
Let $\mu_t$ be defined by \eqref{muttil}, then
$$
\hdd \mu_t = t\alpha(t)+\beta(t).
$$
\end{lem}

\begin{proof}
Similar to Lemma~\ref{lb},  for each integer $i>0$, let $X_i$ and $Y_i$ be the independent random variables on $(\Sigma^\infty,\mathcal{B},\mu_t)$ given by
\begin{eqnarray*}
X_i(\mathbf{w})&=&\log P_i(w_i) - \sum_{w'\in\D_i} P_i(w') \log P_i(w'), \\
Y_i(\mathbf{w})&=&\log Q_i(w_i) - \sum_{w'\in\D_i} P_i(w') \log Q_i(w'),
\end{eqnarray*}
where $ \mathbf{w}=w_1w_2 \ldots w_i \ldots \in \Sigma^\infty.$ By Theorem~\ref{lln}, we have that
$$
\lim_{k\to \infty} \frac{1}{k} \sum_{i=1}^k X_i(\mathbf{w})=0, \qquad \lim_{k\to \infty} \frac{1}{k} \sum_{i=1}^k Y_i(\mathbf{w})=0.
$$
for $\widetilde{\mu}_t$-a.e. $\mathbf{w}\in \Sigma^\infty$.

By the same argument as Lemma~\ref{lb},
we have that
\begin{eqnarray*}
\lim_{k\to \infty}\frac{\log \mu_t(S_k(\mathbf{w}))}{\log r_k}
 &=& \frac{\sum_{\gamma\in\Gamma} f_\gamma\sum_{w\in\D_\gamma}P_{\gamma}(w)\Big(-\zeta\log P_{\gamma}(w) - (1-\zeta)\log Q_{\gamma}(w)\Big)}{\sum_{\gamma\in\Gamma}f_\gamma \log m_\gamma} \\
&=& t\alpha(t)+\beta(t).
\end{eqnarray*}
Then by Lemma~\ref{dimmut}, we have that
$
\hdd \mu_t = t\alpha(t)+\beta(t).
$
\end{proof}

\begin{thm}\label{thmlb}
Let $E$ be the self-affine Moran set with the frequency $\mathbf{f}=\{f_\gamma\}_{\gamma\in\Gamma}$ and $\zeta\leq 1$. Let $\mu$  be the self-affine Moran measure given by \eqref{projmu} .  Then
$$
H(\alpha) \geq \inf_t\{t\alpha+\beta(t)\}.
$$
\end{thm}

\begin{proof}
For each $\alpha\in(\alpha_{\min}, \alpha_{\max})$, by \eqref{def_alpha},  there exists $t>0$ such that $\alpha=\alpha(t)$. Let $E'=\{x\in E_{\alpha(t)} : \lcd \mu(x)=\alpha(t)\}$.
By Lemma~\ref{lem_ld}, $\mu_t(E')=1$. Since
$$
\hdd E_{\alpha(t)} \geq \hdd E' \geq \hdd \mu_t,
$$
by Lemma~\ref{mut}, we have that
$$
\hdd \mu_t = t\alpha(t)+\beta(t) \geq \inf_t \{t\alpha+\beta(t)\},
$$
and the conclusion holds.
\end{proof}

\subsection{Upper Bound}\label{sec_ub}
For $\mathbf{w}\in \Sigma^\infty$ we write
\begin{equation}\label{defB}
I_k(\mathbf{w})=\frac{1}{k} \sum_{i=1}^{k}\log u_i(w_i)
,\qquad
D_k(\mathbf{w})=I_l(\mathbf{w})-I_k(\mathbf{w}),
\end{equation}
where $l=l(k)$ is given by~\eqref{def_l}, and $u_i$ is given by~\eqref{eq_uk}.
We define
\begin{eqnarray*}
&&\hspace{-1cm}\partial E=\left\{x\in E \ \colon \exists \mathbf{w}=(i_1, j_1) (i_1,j_2) \ldots \in \Pi^{-1} (x), \exists K>0, j_k=0, \forall k>K\right\}\\
&&\bigcup \left\{x\in E \ \colon \exists \mathbf{w}=(i_1, j_1) (i_1,j_2) \ldots \in \Pi^{-1} (x), \exists K>0, j_k=m_k-1, \forall k>K\right\}
\end{eqnarray*}

To study the multifractals spectrum on $E$, we need the following technique condition.
\begin{defn}\label{defRC}
We say $E$ satisties \textit{ the replica condition} (RC) if for all $\epsilon>0$ and  all $x \in E\backslash\partial E$, for each $\mathbf{w}\in  \Pi^{-1}(x)$,  there exists a sequence $\{c_{k_i}\}_{i=1}^\infty$, such that
 $$c_{k_i} \geq (m_1 m_{2}\ldots m_{k_i})^{-(1+\epsilon/2)}$$
satisfying that
  $$B(x,c_{k_i})\cap E \subset \bigcup\{S_{k_i}(\mathbf{w}')\colon j_1'=j_1, \ldots, j_{k_i}'=j_{k_i}\}  \qquad \textit{and} \qquad   D_{k_i}(\mathbf{w}) > -\epsilon.$$
\end{defn}

Actually, we prove that $RC$ implies the upper bound. For  $\alpha>0$ and $\epsilon>0$, we write
\begin{eqnarray*}
Y(\epsilon, k)&=&\{S\in \mathcal{S}_k: (m_1\ldots m_{k})^{-\alpha(1+\epsilon)} < \mu(S)<(m_1\ldots m_{k})^{-\alpha(1-\epsilon)}\},\\
G(\epsilon,k) &=& Y(\epsilon,k) \cap \{S_k(\mathbf{w}):D_k(\mathbf{w})\geq -\epsilon, \ \mathbf{w}\in\Sigma^\infty\}.
\end{eqnarray*}
We denote the distance from $x$ to a set $A$ by $d(A,x)=\inf\{|x-y| : y\in A\}$.
\begin{lem}\label{lem12}
Let $E$ be the self-affine Moran set with the frequency $\mathbf{f}=\{f_\gamma\}_{\gamma\in\Gamma}$ and $\zeta\leq 1$.
For all $\epsilon>0$ and  $x\in E_{\alpha}$, there exists $K\in\mathbb{N}$ such that for every $k\geq K$ there exists $\mathbf{w}'\in \Sigma^\infty$ such that $S_k(\mathbf{w}')\in Y(\epsilon, k)$ satisfying
$$
d(S_k(\mathbf{w}'),x)\leq (m_1 m_{2}\ldots m_{k})^{-1}.
$$
\end{lem}

\begin{proof}
  Fix $\epsilon>0$. Since  $x\in E_{\alpha}$, there exists a real $R>0$ such that for all $r<R$,  we have that
  $$
  \alpha\left(1-\frac{\epsilon}{3}\right) \leq \frac{\log \mu(B(x,r))}{\log r} \leq \alpha\left(1+\frac{\epsilon}{3}\right).
  $$
Let $K>0$ be a sufficiently large  integer such that
$$\frac{1}{m_1 m_{2}\ldots m_{K}} < \frac{R}{2C_1} ,  \quad \frac{(2C_1)^{\alpha(1-\frac{\epsilon}{3})}}{  (m_1\ldots m_K)^{\frac{2\alpha\epsilon}{3}}}<1 \quad \textit{and } \quad \frac{(m_1\ldots m_K)^{2\alpha\epsilon/3}}{6N^+}>1,
$$
where $C_1=(1+(N^+)^2)^{1/2}$. For all $k\geq K$ and all $\mathbf{w}'$ such that $d(S_k(\mathbf{w}'),x)\leq (m_1 m_{2}\ldots m_{k})^{-1}$ we have that
\begin{equation}\label{musu}
  \mu(S_k(\mathbf{w}'))\leq (2C_1(m_1 m_{2}\ldots m_{k})^{-1})^{\alpha(1-\frac{\epsilon}{3})}\leq (m_1 m_{2}\ldots m_{k})^{-\alpha(1-\epsilon)}.
\end{equation}

On the other hand, there are at most $6N^+$ approximate squares $S_k(\mathbf{w'})$ such that $d(S_k(\mathbf{w}'),x)\leq (m_1 m_{2}\ldots m_{k})^{-1}$. It follows that at least one of these  approximate squares must satisfies
  $$
  \mu(S_k(\mathbf{w}'))\geq \frac{1}{6N^+}(m_1 m_{2}\ldots m_{k})^{-\alpha(1+\frac{\epsilon}{3})}\geq (m_1 m_{2}\ldots m_{k})^{-\alpha(1+\epsilon)}.
  $$
Therefore $S_k(\mathbf{w}')\in Y(\epsilon, k)$, and  the conclusion holds.
\end{proof}

\begin{lem}\label{cover}
Let $E$ be the self-affine Moran set with the frequency $\mathbf{f}=\{f_\gamma\}_{\gamma\in\Gamma}$ and $\zeta\leq 1$.
Suppose $E$ satisfies the replica condition. Then for all  $\epsilon>0$ and all  $x\in E_{\alpha}$, there exists a sequence $\{k_i\}$ and  \{$\mathbf{w}'_i\in \Sigma^\infty\}$ such that
\begin{itemize}
  \item[(1)] $d(S_{k_i}(\mathbf{w}'_i),x)\leq (m_1 m_{2}\ldots m_{k_i})^{-1}$,
  \item[(2)] $S_{k_i}(\mathbf{w}'_i)\in Y(\epsilon, k_i)$,
  \item[(3)] $D_{k_i}(\mathbf{w}'_i) \geq -\epsilon$, for each  $\mathbf{w}\in \Pi^{-1} (x)$.
  \item[(4)] if $x \notin \partial E$ then $j_1'=j_1, \ldots, j'_{k_i}=j_{k_i}$ for each  $\mathbf{w}\in \Pi^{-1} (x)$.
\end{itemize}
\end{lem}

\begin{proof}
Fix $\epsilon >0$ and $x\in E_\alpha$.  For each $\mathbf{w}\in \Pi^{-1} (x)$, by Lemma~\ref{lem12}, the $(1)$ and $(2)$ hold for all sufficiently large $k$.

For $x \in \partial E$, , we have that
$$
\lim_{k\to\infty} I_k(\mathbf{w})=\sum_{\gamma\in \Gamma}f_\gamma u_{\gamma}(0) \quad \text{or} \quad  \lim_{k\to\infty} I_k(\mathbf{w})=\sum_{\gamma \in \Gamma} f_\gamma u_\gamma(m_\gamma-1).
$$
Thus $\lim_{k\to\infty} D_k(\mathbf{w})=0$, and there exists $K>0$ such that $D_k(\mathbf{w})>-\epsilon$ for $k>K$. Hence the conclusion holds for  $x \in \partial E$.

If $x \notin \partial E$, by replica condition, for each integer $i>0$, the ball $B(x,c_{k_i})$ is contained in the union of at most $N^+$ approximate squares $S_{k_i}(\mathbf{w}')$  satisfying $j'_1=j_1, \ldots, j'_{k_i}=j_{k_i}$, and $D_{k_i}(\mathbf{w}') =D_{k_i}(\mathbf{w}) \geq -\epsilon$. Moreover
$$
\mu(B(x,c_{k_i})) \geq c_{k_i}^{\alpha(1+\frac{\epsilon}{3})} \geq (m_1 m_{2}\ldots m_{k_i})^{-\alpha\left(1+\frac{\epsilon}{2}\right) \left(1+\frac{\epsilon}{3}\right)}.
$$
This implies that the measure on at least one of the approximate squares is no less than $(m_1 m_{2}\ldots m_{k})^{-\alpha(1+\epsilon)}$, and we write a such square as $S_{k_i}(\mathbf{w}'_i)$ with $\mu(S_{k_i}(\mathbf{w}'_i))\geq (m_1\ldots m_{k})^{-\alpha(1+\epsilon)}$.  Together with ~\eqref{musu}, we have that  $S_{k_i}(\mathbf{w}'_i)\in Y(\epsilon, k)$.
\end{proof}

Given $\mathbf{w}\in \Sigma^\infty$, for integers $0<l<k$,we write
$$\Sigma_l^k(\mathbf{w})=\{\mathbf{w}'\in\Sigma^k \colon w_i'=w_i , i=1,2,\ldots, l; j'_i=j_i,i=l+1,\ldots,k\}.$$

\begin{lem}\label{lem_estmu}
For each $\epsilon>0$, there exists an integer $K>0$ such that, for all sufficiently large integer $k$, we have that
\begin{equation*}\label{eqmu1}
  \mu(S_k(\mathbf{w}))^t \leq \left\{\begin{array}{ll}
  e^{2\epsilon k}\sum_{\mathbf{w}'\in \Sigma_l^k(\mathbf{w})}\Big(p_1(w'_1)\ldots p_k(w'_k)\Big)^t \Big(u_1(w'_1)\ldots u_k(w'_k)\Big)^{1-\zeta}, & \zeta <1;\\
 e^{\epsilon k}\Big(p_1(w_1)\ldots p_k(w_k)\Big)^t , & \zeta =1,
 \end{array}\right.
\end{equation*}
 for all $S_k(\mathbf{w})\in G(\epsilon,k)$.
\end{lem}

\begin{proof}
First, we prove  the inequality for $\zeta<1$.
For each $S_k(\mathbf{w})\in \mathcal{S}_k$, by~\eqref{muas} and~\eqref{eq_uk}, we have that
\begin{eqnarray*}
  &&\mu(S_k(\mathbf{w}))^t
  = (p_1(w_1)\ldots p_l(w_l) q_{l+1}(w_{l+1})\ldots q_k(w_k))^t\\
  &=& \Big(p_1(w_1)\ldots p_l(w_l)\Big)^t u_{l+1}(w_{l+1})\Big(\sum_{w_{l+1}'\in D_{l+1} \atop j'_{l+1}=j_{l+1}}p^t_{l+1}(w'_{l+1})\Big)\ldots u_k(w_k)\Big(\sum_{w_{k}'\in D_{k} \atop j'_k=j_k}p^t_k(w'_k)\Big)\\
  &=& \Big(p_1(w_1)\ldots p_l(w_l)\Big)^t u_{l+1}(w_{l+1})\ldots u_k(w_k)\left(\sum_{\mathbf{w}'\in \Sigma_l^k(\mathbf{w})}\Big(p_{l+1}(w'_{l+1})\ldots p_k(w'_k)\Big)^t\right)\\
  &=&\sum_{\mathbf{w}'\in \Sigma_l^k(\mathbf{w})}\Big(p_1(w'_1)\ldots p_k(w'_k)\Big)^t u_{l+1}(w'_{l+1})\ldots u_k(w'_k).
\end{eqnarray*}
  Since $D_k(\mathbf{w})\geq -\epsilon$, by~\eqref{defB}, we have that for $\mathbf{w}'\in \Gamma_k(\mathbf{w})$,
  \begin{eqnarray*}
    \frac{u_{l+1}(w'_{l+1})\ldots u_k(w'_k)}{(u_1(w'_1)\ldots u_k(w'_k))^{1-\zeta}} &=& \frac{(u_1(w_1)\ldots u_k(w_k))^\zeta}{u_1(w_1)\ldots u_l(w_l)}      \\
    &\leq& e^{\zeta k I_k(\mathbf{w})-l I_l(\mathbf{w})}\\
    &\leq& e^{(\zeta k-l)I_k(\mathbf{w})}\cdot e^{\epsilon l}.
  \end{eqnarray*}
  Since $I_k(\mathbf{w})$ is uniformly bounded by some constant $C$ and $\lim_{k\to \infty}\frac{l}{k}=\zeta<1$, there exists $K$ such that for $k\geq K$ we have that $(\zeta k-l)I_k(\mathbf{w})\leq \epsilon l$.
 Hence
  \begin{eqnarray*}
    \frac{u_{l+1}(w'_{l+1})\ldots u_k(w'_k)}{(u_1(w'_1)\ldots u_k(w'_k))^{1-\zeta}}  &\leq& e^{2\epsilon l}\leq e^{2\epsilon k},
  \end{eqnarray*}
which is equivalent to
$$
 u_{l+1}(w'_{l+1})\ldots u_k(w'_k)\leq e^{2\epsilon k} \big(u_1(w'_1)\ldots u_k(w'_k)\big)^{1-\zeta}.
$$
We apply this into the equation of $\mu(S_k(\mathbf{w}))^t$, and the conclusion follows.

  For $\zeta =1$
  Fix $\mathbf{w}$ and k, we have that
  \begin{equation*}
  \mu(S_k(\mathbf{w}))^t=\left\{
  \begin{array}{lcl}
   p_1^t(w_1)\ldots p_l^t(w_l)q_{l+1}^t(w_{l+1})\ldots q_k^t(w_k),  & \  & l \leq k, \\
   p_1^t(w_1)\ldots p_k^t(w_k)q_{k+1}^t(w_{k+1})\ldots q_l^t(w_l), & \  & l > k.
  \end{array}
  \right.
\end{equation*}
Thus
$$
\mu(S_k(\mathbf{w}))^t \leq C^{|(k-l)t|} (p_1(w_1)\ldots p_k(w_k))^t,
$$
where $
C=\max_\gamma \left\{\max_{(i,j)\in \D_\gamma}\left\{\frac{q_{\gamma }(j)}{p_{\gamma}(i,j)}, 2\right\}\right\}.
$

Since $\zeta=1$, $\lim_{k\to \infty}\frac{l}{k}=1 $.  There exists an integer $K>0$ that for $k\geq K$, we have that
$$|k-l| \leq \frac{\epsilon k}{|t|\log C}.$$

Therefore the conclusion holds.

\end{proof}

\begin{thm}\label{thmub}
Let $E$ be the self-affine Moran set with the frequency $\mathbf{f}=\{f_\gamma\}_{\gamma\in\Gamma}$ and $\zeta\leq 1$.  Suppose that $E$ satisfies the replica condition. Then  for each  $\alpha\in(\alpha_{\mathrm{min}},\alpha_{\mathrm{max}})$ we have that
$$
H(\alpha)\leq \inf_t \{\alpha t+\beta(t)\}.
$$
\end{thm}

\begin{proof}
First, we prove the conclusion for $\zeta <1$. Since $E$ satisfies the replica condition, by Lemma~\ref{cover}, we have that for all integer $K\in \mathbb{N}$ and all real $\epsilon>0$,
$$
 E_{\alpha} \subseteq \bigcup_{k>K}\bigcup_{Q_k(\mathbf{w})\in G(\epsilon,k)} \widehat{S}_k(\mathbf{w}),
$$
where $\widehat{S}_k(\mathbf{w})$ is the rectangle with the same centre of $S_k(\mathbf{w})$ but $N^+$ times greater.

Fix a real $t$,
 by \eqref{freq} and  \eqref{betat} , we have that
\begin{eqnarray*}
\lim_{k\to \infty} \Bigg(\frac{m_1^{\beta_1(t)}\ldots m_k^{\beta_k(t)}}{(m_1\ldots m_k)^{\beta(t)}}\Bigg)^\frac{1}{k}
&=& \lim_{k\to \infty} \Bigg(\frac{\prod_{\gamma\in\Gamma} m_\gamma^{c_\gamma(k)\beta_\gamma(t)}}{\big(\prod_{\gamma\in \Gamma}m_\gamma^{c_\gamma(k)} \big)^{\beta(t)}}\Bigg)^\frac{1}{k} = 1.
\end{eqnarray*}
Arbitrarily choose  $\epsilon>0$,  there exists $K_1\in \mathbb{N}$ such that for $k\geq K_1$,
\begin{equation}\label{eqM}
   \frac{m_1^{\beta_1(t)}\ldots m_k^{\beta_k(t)}}{(m_1\ldots m_k)^{\beta(t)}}\leq  (1+\epsilon)^k.
\end{equation}

Let $K_0=\max\{K_1, K_2\}$, where $K_2$ is given by Lemma~\eqref{lem_estmu}.
By Lemma~\ref{lem_estmu} and \eqref{begat}, we have that for $K>K_0$,
\begin{eqnarray*}
&&\sum_{S_k(\mathbf{w})\in G(\epsilon,k)}  (m_1 m_{2}\ldots m_{k})^{-\beta(t)}\mu(S_k(\mathbf{w}))^t \\
&&\leq \sum_{S_k(\mathbf{w})\in G(\epsilon,k)}  (m_1\ldots m_{k})^{-\beta(t)} e^{2\epsilon k}\sum_{\mathbf{w}'\in \Sigma_l^k(\mathbf{w})}\Big(p_1(w'_1)\ldots p_k(w'_k)\Big)^t \Big(u_1(w'_1)\ldots u_k(w'_k)\Big)^{1-\zeta}\\
&&\leq e^{2\epsilon k}  \frac{m_1^{\beta_1(t)}\ldots m_k^{\beta_k(t)}}{(m_1\ldots m_k)^{\beta(t)}}\sum_{S_k(\mathbf{w})\in G(\epsilon,k)}  \sum_{\mathbf{w}'\in \Sigma_l^k(\mathbf{w})} \prod_{i=1}^k\Big(m_i^{-\beta_i(t)}p_i(w'_i)^t u_i(w'_i)^{1-\zeta}\Big)\\
&&\leq e^{2\epsilon k}(1+\epsilon)^k   \sum_{\mathbf{w}'\in \Sigma^k} \prod_{i=1}^k\Big(m_i^{-\beta_i(t)}p_i^t(w'_i) u_i^{1-\zeta}(w'_i)\Big)\\
&&\leq e^{2\epsilon k}(1+\epsilon)^k    \prod_{i=1}^k\sum_{w_i'\in \mathcal{D}_i}\Big(m_i^{-\beta_i(t)}p_i^t(w'_i) u_i^{1-\zeta}(w'_i)\Big)\\
&&\leq  e^{2\epsilon k}(1+\epsilon)^k\\
&&\leq  e^{3\epsilon k}.
\end{eqnarray*}
Let $r_K=3C_1(m_1m_2\ldots m_K)^{-1}$, for integer $K>0$. For all $\delta>\epsilon(\alpha|t|+5)$,  since $m_k\geq 2$, for $K\geq K_0$, we have that
\begin{eqnarray*}
&&\hspace{-2cm}\mathcal{H}_{r_K}^{t\alpha+\beta(t)+\delta}( E_{\alpha}) \leq \sum_{k\geq K} \sum_{S_k(\mathbf{w})\in G(\epsilon,k)} |\widehat{S}_k(\mathbf{w})|^{t\alpha+\beta(t)+\delta}\\
&\leq& (N^+C_1)^{t\alpha+\beta(t)+\delta}\sum_{k\geq K} \sum_{S_k(\mathbf{w})\in G(\epsilon,k)}  (m_1 m_{2}\ldots m_{k})^{-(\beta(t)+5\epsilon)}\mu(S_k(\mathbf{w}))^t\\
&\leq& C_2 \sum_{k\geq K_0} 2^{-5\epsilon k} e^{3\epsilon k} \\
&<& \infty,
\end{eqnarray*}
where $C_1=(1+(N^+)^2)^\frac{1}{2}$, and $C_2=(N^+C_1)^{t\alpha+\beta(t)+\delta}$. This implies that
$$
\hdd  E_{\alpha} \leq t\alpha+\beta(t)+\delta.
$$
Since $\epsilon$ is arbitrarily chosen, $\delta$ can be arbitrarily small, we have that
$$
\hdd  E_{\alpha}\leq t\alpha+\beta(t).
$$
for all $t$.

For the case  $\zeta=1$, the proof is almost identical, and we omit the proof. Therefore the conclusion holds.

\end{proof}

By Theorem~\ref{thmlb} and Theorem~\ref{thmub}, we immediately have the following conclusion.
\begin{thm}\label{thm1}
Let $E$ be the self-affine Moran set with the frequency $\mathbf{f}=\{f_\gamma\}_{\gamma\in\Gamma}$ and $\zeta\leq 1$. Suppose that $E$ satisfies  the replica condition.  Then for every $\alpha\in (\alpha_{\mathrm{min}},\alpha_{\mathrm{max}})$, we have that
$$
H(\alpha)=\inf_t \{\alpha t+\beta(t)\}.
$$
Furthermore, $H(\alpha)$ is differentiable with respect to $\alpha$ and is concave.
\end{thm}

\begin{proof}[The proof of Theorem~\ref{thm_mfa}]
Suppose that $E$ satisfies  row separation condition. For all $\epsilon >0$ and all $x\in E\backslash \partial E$, arbitrarily choose $\mathbf{w}=(i_1,j_1)(i_2,j_2)\ldots(i_k,j_k)\ldots\in \Pi^{-1}(x)$, and let $c_k=\frac{1}{2}(m_1 m_2\ldots m_k)^{-1}$.  it is clear that
  $$
  c_k \geq \big(m_1 m_2\ldots m_k\big)^{-1-\frac{\epsilon}{2}}.
  $$
 Since row separation condition holds, that is to say, for any two rows at $k$-th level intersecting $E$, the distance between the two rows are at least $2c_k$. Therefore, we have that
  $$B(x,c_k)\cap E \subset \bigcup\{S_k(\mathbf{w}') \colon j_1'=j_1, \ldots, j_k'=j_k\}.
  $$
Note that, for all $\mathbf{w}\in \Pi^{-1}(x)$, we have that
\begin{eqnarray*}
  \limsup_{k\to \infty}D_k(\mathbf{w}) &\geq&  \limsup_{l\to \infty}I_l(\mathbf{w})-\limsup_{k\to \infty}I_k(\mathbf{w})= 0.
\end{eqnarray*}
Hence, there exists a sequence $\{k_i\}$ such that $D_{k_i}(\mathbf{w}) > -\epsilon$, and this implies that  $E$ satisfies the replica condition.
By Theorem~\ref{thm1}, conclusion holds.

Suppose $E$ satisfies top and bottom separation condition. There exists $\gamma \in \Gamma$ with $f_\gamma>0$ such that at least one of the following conditions holds:
\begin{itemize}
  \item[(1).]
 For all $(i,j)\in D_\gamma$, $j \neq 0$.
  \item[(2).]
 For all $(i,j)\in D_\gamma$, $j \neq m_\gamma-1$.
\end{itemize}
This implies that either the top row or the bottom row in pattern $(n_\gamma, m_\gamma, \mathcal{D}_\gamma)$ intersecting $E$ is empty.

For all $\epsilon >0$ and all $x\in E\backslash \partial E$, arbitrarily choose $\mathbf{w}\in \Pi^{-1}(x)$ where $\mathbf{w}=(i_1,j_1)(i_2,j_2)\ldots(i_k,j_k)\ldots$. Let $\xi= \frac{\epsilon\log 2}{2 \log N^+}$, and $c_k=(m_1 \ldots m_{[(1+\xi)k]})^{-1} $. For each given $k>0$, since $m_k\geq 2$ , it follows that
\begin{eqnarray*}
c_k &\geq& (N^+)^{-\xi k}(m_1\ldots m_k)^{-1}\geq 2^{-\frac{\epsilon k}{2}} (m_1\ldots m_k)^{-1}\geq (m_1 m_2\ldots m_k)^{-1-\frac{\epsilon}{2}}.
\end{eqnarray*}
Let $\xi'= \frac{\xi}{4+2\xi}$. Since
$$
\lim_{n\to\infty}\frac{\operatorname{card}\{k': (n_{k'},m_{k'},\mathcal{D}_{k'})=\gamma, {k'}=1,2,\ldots, n \}}{n} = f_\gamma > 0,
$$
there exists $K_\xi>0$ such that for $k>K_\xi$,
$$
(1-\xi')f_\gamma<\frac{\operatorname{card}\{k': (n_{k'},m_{k'},\mathcal{D}_{k'})=\gamma, {k'}=1,2,\ldots, n \}}{n}<(1+\xi')f_\gamma.
$$
This implies that for all $k>K_\xi$,
$$
\operatorname{card}\{k':(n_{k'},m_{k'},\mathcal{D}_{k'})=\gamma \text{ for } k<h<[(1+\xi)k]\}\geq 1.
$$
Let $k_0$ be an integer satisfying $k<k_0<[(1+\xi)k]$ and $\D_{k_0}=\D_\gamma$. Then
$$(m_1 \ldots m_{k_0})^{-1} > (m_1 \ldots m_{[(1+\xi)k]})^{-1} = c_k.$$
Since $\D_{k_0}=\D_\gamma$, either the top row or the bottom row of $\D_{k_0}$ intersecting with $E$ is empty, which implies that
$$
B(x,c_k)\cap E \subset B(x,(m_1 \ldots m_{k_0})^{-1})\cap E \subset \bigcup\{S_k(\mathbf{w}') \colon j_1'=j_1, \ldots, j_k'=j_k\}.
$$
Since $\limsup_{k\to \infty}D_k(\mathbf{w}) \geq 0$, there exists a sequence $\{k_i\}$ such that $D_{k_i}(\mathbf{w}) > -\epsilon$. Hence $E$ satisfies the replica condition.
By Theorem~\ref{thm1}, the conclusion holds.

\end{proof}

\begin{proof}[The proof of Theorem~\ref{cor2}]

By Theorem~\ref{thm1}, it is sufficient to prove that $E$ satisfies the replica condition.

Fix $\epsilon>0$ and $x \in E\backslash\partial E$, and choose $\mathbf{w}\in  \Pi^{-1}(x)$. For $c_k=(N^+)^{-2}\big(m_1 m_2\ldots m_k \big)^{-1}$,  we have that, for all sufficiently large $k$,
 $$c_k \geq (m_1 m_2\ldots m_k)^{-1-\frac{\epsilon}{2}}.$$
It remains to show that there exists a sequence $\{k_i\}$ such that
\begin{equation}\label{eqBcover}
B(x,c_{k_i})\cap E \subset \bigcup\{S_{k_i}(\mathbf{w}') \colon  j_1'=j_1, \ldots, j_{k_i}'=j_{k_i}\},
\end{equation}
and $D_{k_i}(\mathbf{w}) > -\epsilon. $
For each integer $k>0$, we write
$$
V_k(\mathbf{w})=\inf\Big\{k'>k: j_{k'}\notin \{0,m_{k'}-1\}\quad \text{or}\quad \frac{j_{k'+1}}{m_{k'+1}-1}\neq \frac{j_{k'}}{m_{k'}-1} \Big\}-k-1.
$$
Simply to say, $V_k(\mathbf{w})$ is the number of $j_{k'}$ constantly chosen from the bottom or constantly chosen from the top in the patterns.
Since $x\notin\partial E$, for each $k$, $V_k(\mathbf{w})< \infty$. By the assumptions (1) and (2) in Theorem~\ref{cor2}, for $k<k'\leq k+V_k(\mathbf{w})$, we have that $u_{k'}(w_{k'})=u_{k+1}(w_{k+1})$, where $j_{k+1}=0 $ or $m_{k+1}-1$, $w_{k+1}=(i_{k+1},j_{k+1})$.

We prove it  by contradiction. We assume that there exists $K>0$ such that for $k>K$, we have that $V_k(\mathbf{w})>0$ or $D_k(\mathbf{w})<-\epsilon$. Note that $V_k(\mathbf{w})=0$ implies either $j_{k+1}\notin\{0, m_{k+1}\}$ or $j_{k+1}\in \{0, m_{k+1}\}$, $\frac{j_{k+2}}{m_{k+2}-1}\neq \frac{j_{k+1}}{m_{k+1}-1} $. Both cases imply ~\eqref{eqBcover}.

Recall that $l(k)$ is a function of $k$ defined by~\eqref{def_l}. To avoid confusion, we write $l_{k+V_k(\mathbf{w})}$  instead of  $l(k+V_k(\mathbf{w}))$  in the following calculation.

For each $k>0$, write $\zeta_k=\frac{l_{k+V_k(\mathbf{w})}}{k+V_k(\mathbf{w})}$. Since $\lim \frac{l}{k}=\zeta$, it is clear that $\lim_{k\to \infty}\zeta_k=\zeta.$ Hence there exists  $C_0>0$ such that $\frac{1-\zeta_k}{\zeta_k}\leq C_0$ for all $k>0$.

Next  we show that $\frac{V_k(\mathbf{w})}{k}$ is bounded by a constant $C$. Suppose $\frac{V_k(\mathbf{w})}{k}>C_0>\frac{1-\zeta_k}{\zeta_k}$. (Otherwise $\frac{V_k(\mathbf{w})}{k}$ is bounded by a constant $C_0$).  Then $k<l_{k+V_k(\mathbf{w})}$. Since $V_{k+V_k(\mathbf{w})}(\mathbf{w})=0$, we have that $D_{k+V_k(\mathbf{w})}(\mathbf{w})<-\epsilon $ and
\begin{eqnarray*}
&&D_{k+V_k(\mathbf{w})}(\mathbf{w})
= I_{l_{k+V_k(\mathbf{w})}}(\mathbf{w}) -I_{k+V_k(\mathbf{w})}(\mathbf{w})\\
&&= \frac{k I_k(\mathbf{w})+ (l_{k+V_k(\mathbf{w})}-k)\log u_{k+1}(w_{k+1})}{l_{k+V_k(\mathbf{w})}}  -
\frac{k I_k(\mathbf{w})+ V_k(\mathbf{w})\log u_{k+1}(w_{k+1})}{k+V_k(\mathbf{w})}\\
&&=\frac{(1-\zeta_k)k}{\zeta_k(k+ V_k(\mathbf{w}))}(I_k(\mathbf{w})-\log u_{k+1}).
\end{eqnarray*}
Since the  set $\Gamma$ is finite, there exists $C_1$, $C_2$ such that $C_1\leq \log u_k\leq C_2$, and $C_1\leq  I_k(\mathbf{w}) \leq C_2$ for all $k$. Hence
$$
\frac{V_k(\mathbf{w})}{k} \leq \frac{\zeta_k-1}{ \epsilon \cdot \zeta_k}\Big(I_k(\mathbf{w})-\log u_{k+1}\Big) -1 \leq \frac{C_0(C_2-C_1)}{\epsilon} -1.
$$
Taking $C=\max\{C_0,\frac{C_0(C_2-C_1)}{\epsilon} -1\}$,  we have that
\begin{equation}\label{VKK}
\frac{V_k(\mathbf{w})}{k} \leq C,
\end{equation}
for all $k\geq K$.

Choose a large integer $h$ such that $h\cdot \epsilon>C_2-C_1$, and choose $K'>K$ such that $\zeta_k>\frac{\zeta}{2}$ and  $\frac{C_2-C_1}{K'}<\frac{\epsilon}{2}$. Thus for $k>K'$,
\begin{equation}\label{bk}
|I_k(\mathbf{w})-I_{k+1}(\mathbf{w})|=\frac{1}{k}|I_{k+1}(\mathbf{w}) -\log u_{k+1}|<\frac{\epsilon}{2}.
\end{equation}
Let $C'=\frac{2(C+1)}{\zeta}$, where $C$ is the constant in~\eqref{VKK}. Since $V_k<\infty$ for all $k>0$.  we choose $n_0>K'(C')^{2(h+1)}$ such that $V_{n_0}(\mathbf{w})=0$.

We inductively  define a sequence $\{n_j\}_{j=1}^{2h}$ such that $D_{n_j}(\mathbf{w}) < -\epsilon$.
Assume that  $D_{n_j}(\mathbf{w}) < -\epsilon$, which implies that
\begin{equation}\label{Iln}
I_{l(n_j)}< I_{n_j}-\epsilon.
\end{equation}
If $D_{l(n_j)}(\mathbf{w}) < -\epsilon$, by setting $n_{j+1}=l(n_j)$, the inequality $D_{n_{j+1}}(\mathbf{w}) < -\epsilon$ holds.
Otherwise if $D_{l(n_j)}(\mathbf{w}) \geq -\epsilon$, we have that $V_{l(n_j)}(\mathbf{w})>0$. Let
 $$
 \lambda_1=\max\{\lambda<l(n_j):V_{\lambda}(\mathbf{w})=0\}\quad \text{and}\quad \lambda_2=\min\{\lambda>l(n_j):V_{\lambda}(\mathbf{w})=0\}.
 $$
 Then we choose
 $$
 n_{j+1}=\left\{\begin{array}{ll}
 \lambda_1, & \text{for}\ I_{\lambda_1}(\mathbf{w})\leq I_{\lambda_2}(\mathbf{w});\\
 \lambda_2, & \text{for}\ I_{\lambda_1}(\mathbf{w})> I_{\lambda_2}(\mathbf{w}).
 \end{array}\right.
 $$
 Since $I_\lambda(\mathbf{w})$ is monotonic for $\lambda_1+1\leq \lambda\leq \lambda_2$, either $I_{\lambda_1+1}(w)$ or $I_{\lambda_2}(w)$ is not greater than $I_{l(n_j)}(\mathbf{w})$. Suppose $n_{j+1}=\lambda_1$. It is clear that $I_{\lambda_1+1}\leq I_{l(n_j)}$. Therefore, we obtain that
 \begin{eqnarray*}
 I_{n_{j+1}}(\mathbf{w}) &\leq& I_{\lambda_1 +1} +\frac{\epsilon}{2} \qquad \qquad \textit{ by } \eqref{bk}\\
 &\leq& I_{l(n_j)}(\mathbf{w})+\frac{\epsilon}{2}\\
  &\leq& I_{n_j}(\mathbf{w})-\frac{\epsilon}{2}.      \qquad\quad \textit{ by } \eqref{Iln}
 \end{eqnarray*}
 See Figure~\ref{figP} for clear relations of these terms.
\begin{figure}[h]
\centering
\includegraphics[width=0.8\textwidth]{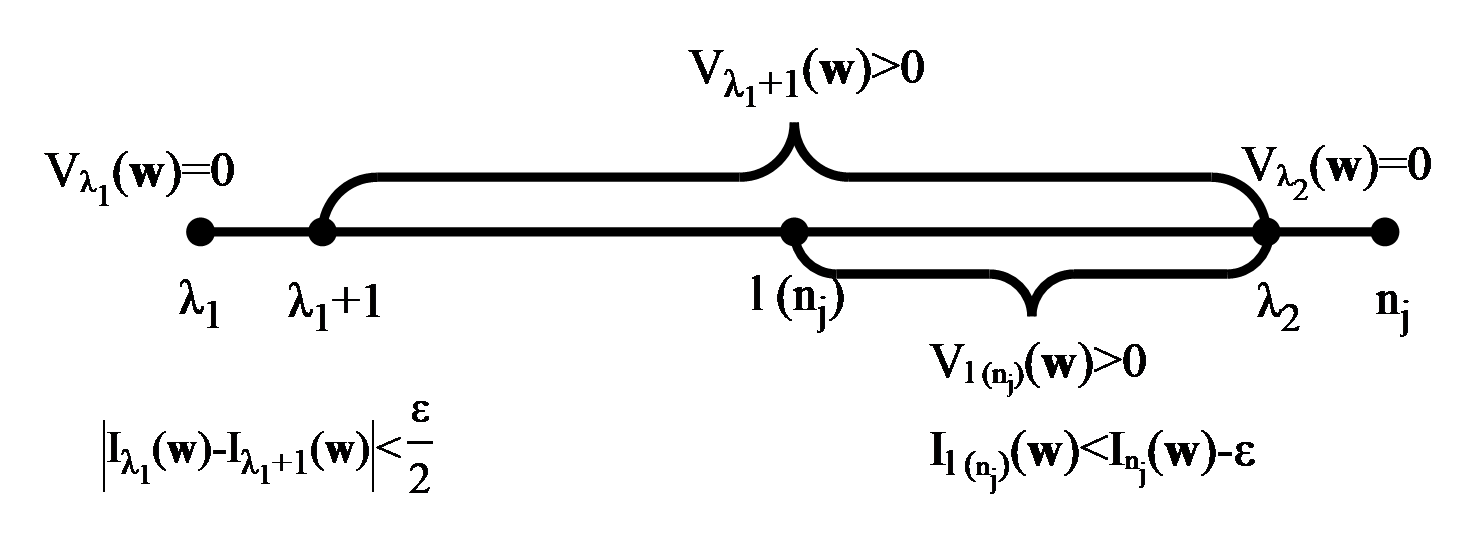}
\caption{}\label{figP}
\end{figure}

Since $l(n_j)\leq n_{j+1}+V_{n_{j+1}}(\mathbf{w})$, by~\eqref{VKK}, we  have that $n_{j+1}\geq\frac{\zeta_{n_j}n_j}{C+1}\geq n_j (C')^{-1}$. Hence  $n_j\geq K'$ for $j=0,1,\ldots,2h$.
Since $I_{n_{j+1}}(\mathbf{w}) \leq I_{n_j}(\mathbf{w})-\frac{\epsilon}{2}$ for each $j$, by the fact $h\cdot \epsilon>C_2-C_1$, we have that
$$
I_{n_{2h}}(\mathbf{w})< I_{n_0}(\mathbf{w})-2h\cdot \frac{\epsilon}{2} \leq C_2-h\epsilon <C_1,
$$
which contradicts to the fact  $C_1\leq I_k(\mathbf{w}) \leq C_2$ for all $k$. Hence $E$ satisfies the replica condition, and by Theorem~\ref{thm1}, the conclusion holds.
\end{proof}


\end{document}